\documentclass[10pt, reqno]{amsart}
\usepackage{amsfonts}
\usepackage{amsbsy}
\usepackage{tikz}
\usepackage{amsmath}
\usepackage{amssymb}
\usepackage{amsthm}
\usepackage{syntonly}
\usepackage{url} 
\usepackage{amscd} 
\usepackage{tikz-cd} 
\usepackage{bm} 
\usepackage{mathdots} 

\usepackage[colorlinks,linkcolor=orange, anchorcolor=green, citecolor=purple ]{hyperref}
\usepackage{graphicx} 

\usepackage{mathrsfs} 
\usepackage{comment} 
\usepackage{bbm} 

\usepackage[utf8x]{inputenc}
    \usepackage{ucs}
\newcommand\quotient[2]{
        \mathchoice
            {
                \text{\raise1ex\hbox{$\#1$}\Big/\lower1ex\hbox{$\#2$}}%
            }
            {
                \#1\,/\,\#2
            }
            {
                \#1\,/\,\#2
            }
            {
                \#1\,/\,\#2
            }
    }

\newcommand{\R}{\mathbb{R}} 
\newcommand{\Z}{\mathbb{Z}}
\newcommand{\Q}{\mathbb{Q}}

\newcommand{\bma}{\bm{\mathrm{a}}}
\newcommand{\bmA}{\bm{\mathrm{A}}}

\newcommand{\bmB}{\bm{\mathrm{B}}}

\newcommand{\bmF}{\bm{\mathrm{F}}}
\newcommand{\bmG}{\bm{\mathrm{G}}}
\newcommand{\bmH}{\bm{\mathrm{H}}}
\newcommand{\bmK}{\bm{\mathrm{K}}}

\newcommand{\bmM}{\bm{\mathrm{M}}}

\newcommand{\bmP}{\bm{\mathrm{P}}}

\newcommand{\bmS}{\bm{\mathrm{S}}}

\newcommand{\bmU}{\bm{\mathrm{U}}}

\newcommand{\bmZ}{\bm{\mathrm{Z}}}

\newcommand{\rmA}{\mathrm{A}}

\newcommand{\rmF}{\mathrm{F}}
\newcommand{\rmG}{\mathrm{G}}
\newcommand{\rmH}{\mathrm{H}}
\newcommand{\rmK}{\mathrm{K}}
\newcommand{\rmL}{\mathrm{L}}
\newcommand{\rmM}{\mathrm{M}}
\newcommand{\rmP}{\mathrm{P}}

\newcommand{\rmS}{\mathrm{S}}
\newcommand{\rmU}{\mathrm{U}}

\newcommand{\rmZ}{\mathrm{Z}}

\newcommand{\fraka}{\mathfrak{a}}

\newcommand{\frakg}{\mathfrak{g}}
\newcommand{\frakh}{\mathfrak{h}}
\newcommand{\frakk}{\mathfrak{k}}

\newcommand{\frakm}{\mathfrak{m}}
\newcommand{\frakp}{\mathfrak{p}}

\newcommand{\fraks}{\mathfrak{s}}

\newcommand{\frakz}{\mathfrak{z}}

\newcommand{\scrG}{\mathscr{G}}
\newcommand{\scrP}{\mathscr{P}}

\newcommand{\scrU}{\mathscr{U}}

\newcommand{\calA}{\mathcal{A}}
\newcommand{\calB}{\mathcal{B}}
\newcommand{\calC}{\mathcal{C}}

\newcommand{\calF}{\mathcal{F}}
\newcommand{\calG}{\mathcal{G}}
\newcommand{\calH}{\mathcal{H}}

\newcommand{\calL}{\mathcal{L}}

\newcommand{\calN}{\mathcal{N}}

\newcommand{\calP}{\mathcal{P}}

\newcommand{\calT}{\mathcal{T}}

\newcommand{\wtX}{\widetilde{X}}
\newcommand{\wtY}{\widetilde{Y}}

\newcommand{\ep}{\varepsilon}

\newtheorem{thm}{Theorem}[section]

\newtheorem{coro}[thm]{Corollary}
\newtheorem{defi}[thm]{Definition}

\newtheorem{lem}[thm]{Lemma}
\newtheorem{prop}[thm]{Proposition}

\newtheorem{ques}[thm]{Question}

\newcommand{\bs}{\backslash}  

\DeclareMathOperator{\SL}{SL}

\DeclareMathOperator{\SO}{SO}

\DeclareMathOperator{\ad}{ad}
\DeclareMathOperator{\Ad}{Ad}

\DeclareMathOperator{\id}{id}

\DeclareMathOperator{\rigid}{rigid}

\DeclareMathOperator{\Tr}{Tr}
\DeclareMathOperator{\Vol}{Vol}

\newcommand{\norm}[1]{\left\lVert#1\right\rVert}

\begin{document}

\title[Uniform nondivergence]{
NONDIVERGENCE ON HOMOGENEOUS SPACES AND RIGID TOTALLY GEODESICS}

\author{Han Zhang}

\author{Runlin Zhang}

\begin{abstract}
Let $G/\Gamma$ be the quotient of a semisimple Lie group by an arithmetic lattice. We show that for reductive subgroups $H$ of $G$  that is large enough, the orbits of $H$ on $G/\Gamma$ intersect nontrivially with a fixed compact set. As a consequence, we deduce finiteness result for totally geodesic submanifolds of arithmetic quotients of symmetric spaces that do not admit nontrivial deformation and with bounded volume. Our work generalizes previous work of Tomanov--Weiss and Oh on this topic.
\end{abstract}

\maketitle
\tableofcontents

\section{Introduction}
\subsection{Main results}
Let $\bmG$ be a semisimple linear algebraic group defined over $\Q$ and $\Gamma \leq \bmG(\Q)$ an arithmetic lattice. Let $\bmH$ be a subalgebraic group of $\bmG_{\R}$ (over $\R$).  
Let the Roman letter $\rmG$ (resp. $\rmH$) denote the identity connected component of $\bmG(\R)$ (resp. $\bmH(\R)$) in the analytic topology. Without loss of generality assume $\Gamma$ is contained in $\rmG$.

\begin{defi}
Let $A$ be a subgroup of $\rmG$, the action of $A$ on $\rmG/\Gamma$ is said to be \textbf{uniformly non-divergent} if there exists a compact set $C\subset \rmG/\Gamma$ such that for all $x \in \rmG/\Gamma$, there exists $g\in A$ such that $g\cdot x \in C$, or equivalently, $\rmG/\Gamma= A \cdot C$.
\end{defi}

We are interested in conditions on $\bmH$ that would guarantee the action of $\rmH$ on $\rmG/\Gamma$ is uniformly non-divergent.

On the one hand, by \cite{DanMar91}, if $\bmH$ is semisimple and has no compact factor, then the action of $\rmH$ on $\rmG/\Gamma$ is uniformly non-divergent if the centralizer of $\bmH$ in $\bmG$ is finite (see \cite[Lemma 3.2]{EinMarVen09}). On the other hand, if $\bmH$ contains a maximal $\R$-split torus, then this is also true and is due to \cite[Theorem 1.3]{TomWei03} extending the idea of Margulis (see \cite[Appendix]{TomWei03}). 

In the present article we find a common generalization of both theorems. 

\begin{thm}\label{thmMain}
Assume that $\bmH$ is connected, reductive, $\R$-split and the centralizer of $\bmH$ in $\bmG$ is $\R$-anisotropic modulo the center of $\bmH$, then the action of $\rmH$ on $\rmG/\Gamma$ is uniformly non-divergent.
\end{thm}

Recall that for an $\R$-linear algebraic group $\bmF$, it is $\R$-anisotropic iff its real point $\bmF(\R)$ is compact.
We shall write $\bmZ_{\bmG}\bmH$ for the centralizer of $\bmH$ in $\bmG$ and $\bmZ(\bmH)$ for the center of $\bmH$.

\subsection{Generalizations}

In this subsection we discuss some further generalization of our main theorems.

Combining with \cite[Corollary 1.3]{ShahWeiss00}, we have
\begin{thm}\label{thmEpi}
Let $\bmF$ be a connected $\R$-subgroup of $\bmG$. Assume that the epimorphic closure of $\bmF$ in $\bmG$ contains a subgroup $\bmH$ satisfying the condition in Theorem \ref{thmMain}. Then the action of $\rmF$ on $\rmG/\Gamma$ is uniformly non-divergent.
\end{thm}

By using \cite[Theorem 1.1]{RicSha18}, we have a uniform version of our main theorem.
\begin{thm}\label{thmUniform}
Consider $\calH$, the set of all connected reductive $\R$-subgroups of $\bmG$ that are $\R$-split and whose centralizer in $\bmG$ is $\R$-anisotropic modulo its center, then we can find a compact set $C\subset \rmG/\Gamma$ such that for all $x\in \rmG/\Gamma$, for all $\bmH\in \calH$, there exists $h\in \rmH$ such that $hx \in C$.
\end{thm}

\subsubsection{Example}
Let $\bmG=\SL_4(\R)$, $\Gamma=\SL_4(\Z)$ and 
\[\bmS:= \left\{\left[
\begin{array}{cccc}
     t^3& && \\
     & t^{-1}&& \\
     & &t^{-1}& \\
     & & &t^{-1}
\end{array}
\right]\right\},\quad
\bmM :=
\left[
\begin{array}{cc}
     1&  \\
     & \bmM_0
\end{array}
\right]
\]
where $\bmM_0$ is $\SO(2,1)\subset \SL_3$. Let $\bmH:=\bmS\cdot \bmH$. Then it is not hard to check that our theorem applies and hence the action of $\rmH$ on $\SL_4(\R)/\SL_4(\Z)$ is uniformly non-divergent. To get a better result, consider 
\[\bmB := \left[
\begin{array}{cc}
     1&  \\
     & \bmB_0
\end{array}
\right]\]
where $\bmB_0$ is a Borel subgroup of $\SO(2,1)$. Let $\bmF:= \bmS \cdot \bmB$, then the epimorphic closure of $\bmF$ in $\SL_4$ is $\bmH$ (because $\bmB_0$ is a parabolic subgroup of $\SO(2,1)$, see \cite{Grosshans97} for details). Hence the action of $\rmF$ on $\SL_4(\R)/\SL_4(\Z)$ is also uniformly non-divergent. Using Lie algebras, it is not hard to show that no proper connected Lie subgroup of $\rmF$ has this property.

\subsection{Geometric consequences}

Let $X$ be an arithmetic quotient of a symmetric space of noncompact type. Let $\calT\calG^{N}$ be the space of embedded totally geodesic finite-volume submanifolds in $X$ of dimension $N$. Equip $\calT\calG^{N}$ with the Chabauty topology. Let $\calT\calG^{N,\rigid}$ be those that do not admit nontrivial deformation in $X$. More precisely, $Y \in \calT\calG^{N,\rigid}$ iff $\{Y\}$ is open in   $\calT\calG^{N}$. We have the following finiteness result generalizing \cite[Theorem 1.1]{OhETDS2004}.

\begin{thm}\label{thmGeometric}
For every natural number $N$ and positive number $T>0$, the set 
\begin{equation*}
     \calT\calG^{N,\rigid}_{\leq T}:=   
    \left\{ Y \in \calT\calG^{N,\rigid}\;\middle\vert\;
    \Vol(Y)\leq T
    \right\}
\end{equation*}
is finite.
\end{thm}

\subsection{Organizations}
In section \ref{secNondivUnimodular} we prove Theorem \ref{thmMain} in the special case of unimodular lattices. The general case is treated in section \ref{sectNondivGeneral}. Theorem \ref{thmUniform} and \ref{thmEpi} are proved in section \ref{secProofThmGeneralized}. In the last section \ref{secGeometric} we prove Theorem \ref{thmGeometric} and in Proposition \ref{propRigidChar} we give a characterization of rigid totally geodesic submanifolds in terms of Lie algebras.

\section{Non-divergence in the space of unimodular lattices}\label{secNondivUnimodular}

In this section we provide a proof of Theorem \ref{thmMain} in the special case when $\bmG=\SL_N$ and $\Gamma$ is commensurable with $\SL_N(\Z)$.  Compared to the general case to be treated in section \ref{sectNondivGeneral}, the proof here is more elementary (though still relies on the non-divergence criterion of Dani--Margulis) and does not rely on \cite{DawGoroUllLi19}. Yet it still illustrates some key ideas also appearing in the general case.

Without loss of generality assume $\Gamma=\SL_N(\Z)$ and 
identify $\SL_N(\R)/\SL_N(\Z)$ as the space of unimodular lattices of $\R^N$.
Fix a reductive subgroup $\bmH$ of $\SL_N$ over $\R$ such that $\bmZ_{\SL_N}\bmH/\bmZ(\bmH)$ is $\R$-anisotropic. Write $\bmH=\bmS\cdot \bmM$ as an almost direct product of an $\R$-split torus and an $\R$-split semisimple group $\bmM$.

Let $\R^N=\bigoplus_{i\in \calA_0} V_i$ be a decomposition of $\R^N$ into $\R$-irreducible representations of $\rmH$.

\begin{lem}
For distinct ${i_1,i_2\in\calA_0}$, $V_{i_1}$ and $V_{i_2}$ are not isomorphic as representations of $\rmH$.
\end{lem}

\begin{proof}
Assume otherwise that $\psi:V_{i_2}\to V_{i_1}$ gives an $\rmH$-equivariant isomorphism. 
Consider  for $s\in \R$
\begin{equation*}
     \left[
    \begin{array}{cc}
       1  & s \psi \\
        0 & 1
    \end{array} \right]
\end{equation*}
as operators on $V_{i_1} \oplus V_{i_2}$. And write $v_s$ for the image of its embedding in $\SL(V)$ by asking that it acts as identity on $V_i$'s for $i\neq i_1, i_2$. Then $\{v_s\}_{s\in\R}$ is a noncompact subgroup of $\rmG$ centralizing $\rmH$, and hence has to be contained in $\rmH$. But it does not preserve $V_{i_2}$, which is a contradiction.
\end{proof}

In light of this lemma, $\rmS$ can be described more concretely. On the one hand, every $s\in \rmS$ acts as a positive scalar $s_i$ when restricted to each $V_i$. On the other hand, if $s\in\rmG$ acts as a positive scalar when restricted to each $V_i$, then $s\in \rmS$ because it centralizes $\rmH$, is $\R$-diagonalizable and also is connected to the identity via some one-parameter flow.


Let $\norm{\cdot}$ be the standard Euclidean metric on $\R^N$ and by abuse of notation also the induced metrics $\wedge^i\R^N$ for all $i$'s.
For a lattice $\Lambda \leq \R^N$, an $\R$-linear subspace (will be abbreviated as an $\R$-subspace) $W$ of $\R^N$ is said to be $\Lambda$-rational iff $\Lambda\cap W \leq W$ is a lattice, in which case we let $\Lambda_W:=\Lambda\cap W \leq W$ and let $\norm{\Lambda_W}$ denote the volume of $W/\Lambda_W$.
If $v_1,...,v_k$ is a set of $\Z$-basis of $\Lambda_W$ then $\norm{\Lambda_W}= \norm{v_1\wedge...\wedge v_k}$.
A subspace $W$ is said to be $(\rmM,\Lambda)$-\textbf{eligible} iff $W$ is both $\rmM$-stable and $\Lambda$-rational.

Let $\delta_{\rmM}: \rmM\backslash \SL_N(\R)/\SL_N(\Z) \to (0,\infty)$ be defined by 
\begin{equation*}
    \delta_{\rmM}([\Lambda]):= 
    \inf \left\{
     \norm{\Lambda_W}^{\frac{1}{\dim W}} \:\middle\vert\:
     W \text{ is } (\rmM,\Lambda)\text{-eligible}
    \right\}.
\end{equation*}
When $\bmM=\bmH$, $\delta_{\rmM}([\Lambda])$ is always equal to $1$. When $\rmM=\{e\}$ is the trivial group, write $\delta:=\delta_{\rmM}$. By Mahler's criterion, to prove Theorem \ref{thmMain}, it suffices to show that there exists $\eta>0$ such that for every $\Lambda\in \SL_N(\R)/\SL_N(\Z)$, there exists some $h\in \rmH$ such that $\delta(h\Lambda)>\eta$. This would in turn follow from the following proposition by \cite[Theorem 4.6]{DawGoroUll18} which is based on the work of \cite{DanMar91} (see also \cite[Corollary 3.3, Theorem 3.4]{Kle10}).

\begin{prop}\label{prop_pushoutSL_N}
There exist $0<\eta_0<1$ and $C>1$ such that for all $[\Lambda]\in  \rmM\backslash \SL_N(\R)/\SL_N(\Z) $ with $\delta_{\rmM}([\Lambda])<\eta_0$, there exists $s\in S$ such that $\delta_{\rmM}([s\Lambda])\geq C \delta_{\rmM}([\Lambda])$. As a result, there exists $s\in S$ such that $\delta_{\rmM}([s\Lambda])\geq \eta_0$.
\end{prop}

The proof of this proposition will be based on the two lemmas below.

\begin{lem}\label{lemExpand_SL_N}
There exist, and we fix, $C_1,C_2>1$ and a finite subset $\calF\subset S$ such that for every proper $\R$-subspace $W$ of $\R^N$ that is $\rmM$-stable, there exists $s\in \calF$ such that 
\begin{enumerate}
    \item $\norm{sv}>\frac{1}{C_1}\norm{v}$ for all pure wedges $v$ in $\R^N$;
    \item $\norm{sv}>C_2\norm{v}$ for all pure wedges $v$ with $\calL_v$ contained in $W$.
\end{enumerate}
\end{lem}

By a pure wedge $v$, we mean some non-zero vector of the form $v=v_1\wedge...\wedge v_i$ in $\wedge^i\R^N$ for some $i$. For such a pure wedge, write $\calL_v$ for the $\R$-subspace of $\R^N$ spanned by $v_1,...,v_i$.

\begin{proof}
The first part comes for free as long as $\calF$ is a finite set. We shall focus on the second part.
It suffices to show that for each $\rmM$-stable subspace $W$, there exists $s_W\in S$ such that
\[
\norm{s_w\cdot w} > \norm{w},\quad \forall w\in W
\]
Then the same thing would be true replacing $w$ by any pure wedge $w$ with $\calL_w$ contained in $W$. A continuity argument applied to the unit vectors would then finish the proof.

Recall the decomposition of $\R^N= \bigoplus_{i\in \calA_0} V_i$ into irreducible representations with respect to $\rmH$. 
Also, for each $i$, every $s\in \rmS$ acts as $s_i\text{id}_{V_i}$, for some $s_i>0$, when restricted to $V_i$ .

For each $I\subset \calA_0$, define $\pi_I:\R^N \to V_I:=\bigoplus_{i\in I} V_i$ to be the corresponding  $\rmH$-equivariant projection. We claim that for any $\rmM$-stable $W$, there exists $I$ such that $\pi_I\vert_{W}$ is bijective.

Fix W and we define $I=\{i_1,...,i_k\}$ inductively. Firstly we pick $i_1\in\calA_0$ such that $\pi_{i_1}\vert_{W}\neq 0$ and let $W_1:=\ker (\pi_{i_1}\vert_{W})$. Secondly we pick $i_2\in\calA_0$ such that $\pi_{i_2}\vert_{W_1}\neq 0$ and let  $W_2:=\ker (\pi_{i_2}\vert_{W_1})$. Continue this until $\pi_i\vert_{W_k}=0$ for all $i\in\calA_0$, in which case $W_k=\{0\}$. So for $j\leq k$ we have the exact sequence
\begin{equation*}
   \begin{tikzcd}
        0\arrow[r]& W_{j_0}\arrow[r]& W_{j-1}\arrow[r] & \pi_{j_0}(W_{j-1})\arrow[r] &0.
   \end{tikzcd}
\end{equation*}

Now $\pi_I\vert_{W}$ is injective. Indeed if $w\in W$ is such that $\pi_I(w)=0$, then $w\in\cap_i W_i =\{0\}$.

The map $\pi_I\vert_{W}$ is also surjective. Write $W_0:=W$.
For each $i=1,...,k$, 
\[
\pi_i\vert_{W_{i-1}} \neq 0 \implies 
\pi_i\vert_{W_{i-1}} \text{ is surjective on }V_i 
\]
as $V_i$ is actually an irreducible representation with respect to $\rmM$ (though $V_i$ may be isomorphic to $V_{j_0}$ as $\rmM$-representations).  As $W_{i-1}$ is defined to vanish under $\pi_{i-1}$ (for $i>1$), an induction argument shows that $\pi_{\{i_1,...,i_l\}}\vert_{W}$ is surjective for $l=1,...,k$.

Once $\pi_I\vert_{W}:W\to V_I$ is bijective, we see that $W$ is the graph of some linear map $\phi_W:V_I \to V_{I^c}$. That is to say, for any $w\in W$, there exists a unique $v\in V_I$ such that $w=v+\phi_W(v)$.
Find $C_W>0$ such that 
\begin{equation*}
    \begin{aligned}
    &\norm{v+\phi_W(v)} \leq C_W \norm{v},\quad \forall v\in V_I.\\
    & \frac{1}{C_W}\norm{v} \leq \norm{v+v'} ,\quad \forall v\in V_I, \;v'\in V_{I^c}.\\
    \end{aligned}
\end{equation*}
Then we take $s_W\in S$ such that $s_w \vert_{V_I} =\lambda \id_{V_I} $ for some $\lambda>C_W^2$. Thus
\[
\norm{s_w\cdot w} \geq \frac{1}{C_W} 
\norm{\lambda v }
> C_W \norm{v} \geq \norm{v + \phi_W(v)} = \norm{w}.
\]
Now the proof is complete.
\end{proof}

\begin{lem}\label{lemProtect_SL_N}
For every $[\Lambda] \in \rmM\backslash \SL_N(\R)/\SL_N(\Z)$ with $\delta_{\rmM}([\Lambda])<\eta_0:=\frac{1}{(C_1C_2)^{N}}$, there exists a proper $(\rmM,\Lambda)$-eligible subspace $W_{\infty}\leq \R^N$ such that for any $(\rmM,\Lambda)$-eligible subspace $W$ not contained in $W_{\infty}$ we have 
\begin{equation*}
    \norm{\Lambda_W+\Lambda_{W_{\infty}}} \geq (C_1C_2)\norm{\Lambda_{W_{\infty}}}.
\end{equation*}
\end{lem}

\begin{proof}

Take such a $\Lambda$ as in the statement. Find an $(\rmM,\Lambda)$-eligible $\R$-subspace $W_1$  such that $\norm{\Lambda_{W_1}}<\eta_0$. If $W_1$ satisfies the conclusion above, then we take $W_{\infty}=W_1$. Otherwise there exists $W_1'\leq \R^N$, $(\rmM,\Lambda)$-eligible, such that 
\begin{equation*}
    \norm{\Lambda_{W'_1}+\Lambda_{W_{1}}} < (C_1C_2)\norm{\Lambda_{W_1}}.
\end{equation*}
Let $W_2:=W_1'+W_1$, then $W_2$ is still $(\rmM,\Lambda)$-eligible. If $W_2$ satisfies the conclusion, we stop. Otherwise we define $W'_2$ and $W_3$ similarly as above. As the dimension is finite, this process has to stop at some index $l$ not exceeding $N$. We let $W_{\infty}:=W_l$ and it only remains to show that $W_{\infty}$ is a proper subspace. Indeed $\Lambda_{W_i}+\Lambda_{W_i'}\leq \Lambda_{W_{i+1}}$, so
\begin{equation*}
    \norm{\Lambda_{W_l}} \leq \norm{\Lambda_{W'_{l-1}} + \Lambda_{W_{l-1}}}  \leq (C_1C_2)\norm{\Lambda_{W_{l-1}}} \leq ...\leq (C_1C_2)^{l} \norm{W_1} < 1.
\end{equation*}
As $\norm{{\R^N}}_{\Lambda}= \norm{\Lambda}=1$, $W_{\infty}\neq \R^N$. 
\end{proof}

Now finally we come to the

\begin{proof}[Proof of Proposition \ref{prop_pushoutSL_N} with $C:=C_2$]

Recall $\eta_0=1/(C_1C_2)^N$.
We take $[\Lambda]\in \rmM\bs\SL_N(\R)/\SL_N(\Z)$ such that $\delta_{\rmM}([\Lambda])<\eta_0$. By Lemma \ref{lemProtect_SL_N}, pick  a proper $(\rmM,\Lambda)$-eligible subspace $W_{\infty}\leq \R^N$ such that for any $(\rmM,\Lambda)$-eligible subspace $W$ not contained in $W_{\infty}$ we have 
\begin{equation*}
    \norm{\Lambda_W+\Lambda_{W_{\infty}}} \geq (C_1C_2)\norm{\Lambda_{W_{\infty}}}.
\end{equation*}
Then by applying Lemma \ref{lemExpand_SL_N} to $W_{\infty}$, we get some 
$s\in S$ such that 
\begin{enumerate}
    \item $\norm{sv}>\frac{1}{C_1}\norm{v}$ for all pure wedges $v$ in $\R^N$;
    \item $\norm{sv}>C_2\norm{v}$ for all pure wedges $v$ with $\calL_v$ contained in $W_{\infty}$.
\end{enumerate}

Now we prove our assertion with this $s\in S$.
It suffices to show that for every $(\rmM,s\Lambda)$-eligible $W'$, we have $\norm{\Lambda_{W'}}\geq C_2 \norm{\Lambda_{W''}}$ for some $(\rmM,\Lambda)$-eligible $W''$.

First let $W:=s^{-1}W'$, then $W$ is $(\rmM,\Lambda)$-eligible and $s\cdot \Lambda_W= (s\Lambda)_{W'}$. There are two cases to consider.

\textit{Case I}, $W\subset W_{\infty}$. Then
\begin{equation*}
    \norm{(s\Lambda)_{W'}} =
    \norm{s\Lambda_W} \geq C_2 \norm{\Lambda_W}.
\end{equation*}
So setting $W'':=W$ concludes the proof.

\textit{Case II}, $W\nsubseteq W_{\infty}$. Let $W'':=W\cap W_{\infty}$. Then $W''$ is $(\rmM,\Lambda)$-eligible and $\Lambda_{W''}=\Lambda_W\cap \Lambda_{W_{\infty}}$.
We have
\begin{equation*}
\begin{aligned}
        &\norm{\Lambda_{W''}} (C_1C_2) \norm{\Lambda_{W_{\infty}}}
    \leq \norm{\Lambda_W \cap \Lambda_{W_{\infty}}} \cdot \norm{\Lambda_W + \Lambda_{W_{\infty}}} 
    \leq \norm{\Lambda_W} \cdot \norm{\Lambda_{W_{\infty}}} \\
    &\implies \norm{\Lambda_W} \geq (C_1C_2) \norm{\Lambda_{W''}}.
\end{aligned}
\end{equation*}
So
\begin{equation*}
    \norm{(s\Lambda)_{W'}} \geq
    \frac{1}{C_1} \norm{\Lambda_W}
    \geq C_2 \norm{\Lambda_{W''}}
\end{equation*}
and we are done.

\end{proof}

\section{Non-divergence in the general case}\label{sectNondivGeneral}
In this section we prove Theorem \ref{thmMain} in general.

So let $\bmG$ be a  $\Q$-semisimple group of dimension $N$ and $\Gamma \leq \rmG$ be an arithmetic lattice. Let $\bmH \leq \bmG$ be a connected reductive subgroup defined over $\R$ without compact factors.  
Write $\bmH=\bmS\cdot \bmM$ as an almost direct product of some $\R$-split torus $\bmS$ and some connected $\R$-split semisimple group $\bmM$.
We assume that 
$\bmS\leq \bmZ_{\bmG}\bmM$ is a maximal $\R$-split torus.

Fix a maximal compact subgroup  $\rmK$ of $\rmG$ and an $\Ad(\rmK)$-invariant metric on $\scrG$, defined as the Lie algebra of $\rmG$. Only in this subsection we follow the convention of \cite{TomWei03} to use script letters for Lie algebras.
We fix an integral structure $\scrG_{\Z}$ on $\scrG$ that is contained in the Lie algebra of $\bmG$ and is preserved by $\Ad(\Gamma)$. 
For each $g\in \rmG$, write $\scrG_{g}:= \Ad(g)\cdot \scrG_{\Z}$.
For $\eta>0$, let $\calN_{\eta}:=\{v\in    \scrG\,\vert\, \norm{v}<\eta\}$.
For a discrete subgroup $\Lambda$ of $\scrG$, we let $\norm{\Lambda}$ be the covolume of $\Lambda$ in $\Lambda_{\R}$, the $\R$-span of $\Lambda$.

For each $\eta>0$, let $X_{\eta}$ be a compact subset of $\rmG/\Gamma$ defined by 
\begin{equation*}
    X_{\eta}=\left\{
    [g]\in\rmG/\Gamma\:\vert\:
    \scrG_{g} \cap \calN_{\eta} =\{0\}
    \right\}.
\end{equation*}
As the map $[g]\mapsto \scrG_{g}$ is a proper map from $\rmG/\Gamma$ to the space of lattices in $\scrG$ with some fixed volume, the union of interiors of $\{X_{\eta}\}_{\eta>0}$ covers $\rmG/\Gamma$ by Mahler's criterion.

To take into consideration of $\rmM$, we define 
\begin{equation*}
    X^{\rmM}_{\eta}=\left\{
    [g]\in\rmG/\Gamma\:\vert\:
    \rmM[g]\cap X_{\eta}\neq \emptyset
    \right\}.
\end{equation*}

We need to introduce some terminologies from reduction theory. For more detailed expositions one may consult \cite{Bor2019} (see also \cite{BorSer73, BorJi06, DawGoroUll18, ZhangArxivBorelSerre}).

There exists a finite collection of  $\Q$-parabolic subgroups $\{\bmP_i\}_{i \in \calA_1}$ of $\bmG$ such that any $\Q$-parabolic subgroup $\bmP$ of $\bmG$ is conjugate to one of $\bmP_i$ by $\Gamma$.
Let $\bmU_i$ be the unipotent radical of $\bmP_i$, then $\bmP_i/\bmU_i$ is a $\Q$-reductive group. Let $\bmS'_i$ denote the $\Q$-split part of its center. The lift of $\bmS'_i$ to $\bmP_i$ is not unique, but we fix one $\bmS_i$ that is defined over $\Q$.
On the other hand we take another lift $\rmA_i$ of $\rmS'_i$ that is invariant under the Cartan involution on $\rmG$ associated with $\rmK$. We let $\Delta_i$ be the simple roots for $(\rmA_i,\bmP_i)$. As $\rmA_i$ is conjugate to $\rmS_i$ in a unique way, we are safe to think of $\Delta_i$ also as simple roots for $(\bmS_i,\bmP_i)$, in which case consists of $\Q$-characters.
For a $\Q$-parabolic subgroup $\bmP$, let ${}^{\circ}\!\bmP$ be the subgroup of $\bmP$ defined by the common kernel of all $\Q$-characters of $\bmP$. And let ${}^{\circ}\!\rmP$ be the identity connected component, in the analytic topology, of ${}^{\circ}\!\bmP(\R)$. Associated with $(\rmK, \bmP_i)$, we write $g=k^i_g a^i_g p^i_g$ for the horospherical coordinates of $g\in \rmG$ (see for instance \cite[Section 2.3]{ZhangArxivBorelSerre}, one should take inverse of everything happening in the reference and combine the $\rmM$, $\rmU$ term together to get $p_g\in {}^{\circ}\!\rmP$). Note $k^i_g \in \rmK$, $a^i_g\in \rmA_i$ and $p^i_g \in {}^{\circ}\!\rmP$.

Now we define generalized Siegel sets taking into considerations of $\rmM$. When $\rmM=\{e\}$, this specializes to the usual Siegel set.
For each index $i\in \calA_1$, a bounded set $B\subset {}^{\circ}\!\rmP$ and $\theta,\ep>0$,
define
\begin{equation*}
    \Sigma^{\rmM}_{i,B,\theta}:=
    \left\{
    g\in \rmG \;\middle\lvert\;
    \parbox{50mm} 
    {
    \raggedright  $g^{-1}\rmM g \subset \rmP_i; \,
    \alpha(a^i_g)<\theta,\, \forall \alpha\in \Delta_i; \,\exists m\in \rmM,\,
    p^i_g g^{-1}mg \in B\Gamma_{ {}^{\circ}\!\rmP}$
    }
    \right\}.
\end{equation*}
and 
\begin{equation*}
    \Sigma^{\rmM}_{i,B,\theta,\ep}:=\left\{
    g\in  \Sigma^{\rmM}_{i,B,\theta} \:\middle\lvert\:
    \alpha(a^i_g)<\ep,\:\exists \alpha \in \Delta_i
    \right\}.
\end{equation*}

We need the following proposition, which is a corollary to the main result of \cite{DawGoroUllLi19}.

\begin{prop}\label{propDGUL}
For each $0<\theta<1$, there exist $B\subset \rmG$ bounded and $\eta'>0$ such that 
\begin{enumerate}
    \item 
    for every $g\in \rmG$,
    \[
    g\notin X^{\rmM}_{\eta'} \implies 
    g\in \bigcup_i \Sigma^{\rmM}_{i,B\cap {}^{\circ}\!\rmP_i,\theta} \Gamma;
    \]
    \item fix such a set of $\theta=\theta_0$, $B=B_0$ and $\eta'=\eta_0$, then there exists a function $\ep_0:(0,\eta_0)\to (0,\infty) $ with $\lim_{\eta\to0} \ep_0(\eta)=0$ such that 
    for every $g\in \rmG$,
    \[
     g\notin X^{\rmM}_{\eta} \implies 
    g\in \bigcup_i \Sigma^{\rmM}_{i,B\cap {}^{\circ}\!\rmP_i,\theta,\ep_0(\eta)} \Gamma.
    \]
\end{enumerate}
\end{prop}

\begin{proof}[Proof of (1)]
For each positive integer $n$, we let $B_n:=X_{\frac{1}{n}}$ and  $\eta_n:=\frac{1}{n}$.
If (1) were not true, then there exists $\theta_0\in(0,1)$ and a sequence of $g_n \notin X^{\rmM}_{\eta_n}$, yet $g_n \notin \bigcup_i \Sigma^{\rmM}_{i,B_n\cap {}^{\circ}\!\rmP_i,\theta_0} \Gamma$.
By definition of $X^{\rmM}_{\eta_n}$, $(\rmM[g_n])$ diverges topologically in $\rmG/\Gamma$.

By \cite[Section 5.3]{DawGoroUllLi19} (where they proved the hypothesis of \cite[Theorem 4.2]{DawGoroUllLi19} is met) and after passing to a subsequence, there exists $i_0\in\calA_1$ and $\gamma_n\in \Gamma$ such that $M_n:=\gamma_n^{-1}g^{-1}\rmM g\gamma_n$ is contained in ${}^{\circ}\!\rmP_{i_0}$ and if we write 
\[
g_n\gamma_n =k_n a_n p_n, 
\]
the horospherical coordinate of $g_n\gamma_n$ with respect to $\bmP_{i_0}$ and $\rmK$, then
\begin{enumerate}
    \item $d_n:= \max_{\alpha\in \Delta_{i_0}}\alpha(a_n) \to 0$;
    \item $(p_nM_n)$ is non-divergent in ${}^{\circ}\!\rmP_{i_0}/{}^{\circ}\!\rmP_{i_0}\cap \Gamma$. That is to say, there exist a sequence $(m_n)$ in $\rmM$, a bounded sequence $(b_n)$ in ${}^{\circ}\!\rmP_{i_0}$ and $(\lambda_n)$ in ${}^{\circ}\!\rmP_{i_0}\cap \Gamma$ such that 
    \[
    p_n \gamma_n^{-1}g_n^{-1} m_n g_n \gamma_n = b_n \lambda_n.
    \]
\end{enumerate}
Now we let $B:=\{b_n\}$, which is a bounded set. Then we have
\[
g_n\gamma_n \in \Sigma^{\rmM}_{i_0,B,d_n}, \quad \text{so}\quad
g_n\in \Sigma^{\rmM}_{i_0,B,d_n} \Gamma.
\]
When $n$ is large enough such that $B$ is contained in $B_n\cap {}^{\circ}\!\rmP_{i_0}$ and $d_n <\theta_0$, we have a contradiction.
\end{proof}

\begin{proof}[Proof of (2)]
By (1) for each $g\in \rmG$ with $[g]\notin X^{\rmM}_{\eta_0}$, choose $i_g\in \calA_1$ such that $g\in \Sigma^{\rmM}_{i_g,B_0,\theta_0}\Gamma$. The choice of $i_g$ may not be unique, but we just fix one.

Define for  $g\in \rmG$ with $[g]\notin X^{\rmM}_{\eta_0}$,
\[
\ep_0(g):= \inf \left\{
\ep>0 \:\middle\vert \:
g\in \Sigma_{i_g,B_0,\theta_0,\ep}\Gamma
\right\}.
\]
From the definition we see that $0<\ep_0(g)\leq \theta_0$ and $g\in \Sigma_{i_g,B_0,\theta_0,2\ep_0(g)}\Gamma$. So (2) amounts to saying that (there exists some choice of $i_g$ such that) 
\[
\lim_{\eta\to0} \sup_{[g]\notin X^{\rmM}_{\eta}}\ep_0(g) =0.
\]

Hence if (2) is not true, then there exists $\ep_0>0$ and a sequence $g_n \notin X^{\rmM}_{\eta_n}$ such that $\ep_0(g_n)\geq \ep_0$ for all positive integers $n$. 
By passing to a subsequence we may assume that $i_{g_n}$ are identically equal to some $i_0\in\calA_1$ for all $n$. 
So there exists $\gamma_n \in \Gamma$ such that if $g_n\gamma_n=k_na_np_n$ is the horospherical coordinate of $g_n\gamma_n$ with respect to $\bmP_{i_0}$ then $(k_na_n)$ is bounded by the assumption that $\ep_0(g_n)\geq \ep_0$.
Also $p_n\gamma_n^{-1}g_n^{-1} \rmM g_n\gamma_n \cap B_0\Gamma\neq \emptyset$. So there exist a bounded set $B\subset\rmG/\Gamma$ such that $\rmM [g_n]\cap B\neq \emptyset$ for all $n$. This contradicts against $g_n \notin X^{\rmM}_{\eta_n}$ for some $\eta_n \to 0$.
\end{proof}

From now on we fix a choice of $\theta_0,\eta_0 \in(0,1)$, $B_0$ bounded in $\rmG$ and $\ep_0:(0,\eta_0)\to(0,\infty)$ satisfying the above proposition. 
We choose $\theta_0>0$ small enough such that 
\begin{itemize}
    \item $\alpha(a_g)<1$ for all $g\in \bigcup_i\Sigma^{\rmM}_{i,B_0\cap {}^{\circ}\!\rmP_i,\theta_0}$ and all $\alpha\in \Phi_i^{-}$
\end{itemize}
where $\Phi_i^-$ denotes all nontrivial characters of $\rmA_i$ that appears in the Lie algebra of $\rmP_i$. Elements of $\Phi_i^-$ are positive linear combinations of those from $\Delta_i$, thus such a choice of $\theta_0$ exists.
By choosing a smaller $\eta_0$, we assume that $0<\ep_0(\eta)<1$ for all $0<\eta<\eta_0$.

Define a function $ \delta_{\rmM}: \rmM\backslash\rmG/\Gamma \to (0,\infty)$ by
\begin{equation*}
   \delta_{\rmM}([g]):= \inf
   \left\{
    \norm{\calL\cap \scrG_g }^{\frac{1}{\dim\calL}} \,\middle\vert\,
    \calL\leq \scrG \text{ is } \scrG_g\text{-rational and } \rmM\text{-stable} 
   \right\}.
\end{equation*}

For each $i\in \calA_1$, 
fix (the unique) $\omega_i\in \rmU_i$ such that $\omega_i \rmA_i\omega_i^{-1}= \rmS_i$ and
decompose $\scrG$ according to the Adjoint action of $\rmA_i$, $\rmS_i$:
\[
\scrG=\bigoplus_{\alpha\in \Phi_i(\rmS_i)} \scrG^{\rmS_i}_{\alpha},
\quad
\scrG=\bigoplus_{\alpha\in \Phi_i(\rmA_i)} \scrG^{\rmA_i}_{\alpha}.
\]

We identify $\Phi_i(\rmS_i)$ with $\Phi_i(\rmA_i)$ via $\Ad(w_i)$ and will simply refer to them as $\Phi_i$. 
By definition non-zero weights appearing in the Lie algebra of $\rmP_i$, or equivalently in $\scrU_i$, the Lie algebra of $\rmU_i$, have been called negative, and
write $\Phi_i^{*}$ for the negative, zero or positive weights for $*=-,0,+$ respectively.
Also $\Phi_i^{0-}:=\Phi_i^{-}\sqcup \Phi_i^{0}$ and $\Phi_i^{0+}:=\Phi_i^{+}\sqcup \Phi_i^{0}$.
For each $\alpha\in\Phi_i$, let $\pi^{\rmS_i}_{\alpha}$ and $\pi^{\rmA_i}_{\alpha}$ denote the corresponding projections to the weight space. Note $\Ad(w_i)(\scrG^{\rmA_i}_{\alpha}) = \scrG^{\rmS_i}_{\alpha}$, so we have
\[
\Ad(w_i)\circ \pi^{\rmA_i}_{\alpha}
=\pi^{\rmS_i}_{\alpha} \circ \Ad(w_i).
\]

Also for each $i\in\calA_1$, define
$\pi^{\rmA_i}_{0+}:\scrG \to \bigoplus_{\alpha\in \Phi^{0+}_i(\rmS_i)} \scrG^{\rmS_i}_{\alpha}$  to be the natural projection 
and similarly define $\pi^{\rmA_i}_{-}$, $\pi^{\rmS_i}_{0+}$ and $\pi^{\rmS_i}_{-}$. They are related in the same manner as above.

Note that $\pi^{\rmA_i}_{0+}$ is also the orthogonal projection onto $\scrU_i^{\perp}$, which is denoted by $\pi_{\scrU_i^{\perp}}$. 
Actually, the usefulness of $\rmA_i$ comes from the fact it is invariant under the Cartan involution and hence $\pi_{\alpha}^{\rmA_i}$'s are all orthogonal projections whereas the rationality of $\rmS_i$'s makes $\pi_{\alpha}^{\rmS_i}$'s defined over $\Q$.

To be prepared for the upcoming corollary, we define some constants. For each $i\in \calA_1$ and $b\in B_0\cap \rmP_i$, let $h^i_b \in Z_{\rmG}(\rmA_i)$ and $u^i_b \in \rmU_i$ such that $b=h^i_b u^i_b$. Then the set of all possible $\{h^i_b\}$ as $i$ and $b$ vary is also bounded. Define
\[
\begin{aligned}
&C_3:= \max\left\{
1, \norm{\scrU_i\cap{\scrG_{\Z}}}
\;\middle\vert\;
i\in \calA_1
\right\}\\
&C_4:= \sup \left\{
1, \, \norm{\Ad(\omega_i)^{-1}},\,
\norm{\Ad(\omega_i)},\, 
\norm{ \Ad(\omega_ih^i_b\omega_i^{-1})^{-1}}
\;\middle\vert\;
i\in \calA_1,\,b\in B_0
\right\}.
\end{aligned}
\]
where $\norm{\Ad(g)}$ denotes the operator norm of $\Ad(g)$.

We also choose $C_5>1$ such that  for each $i\in \calA_1$,
\begin{itemize}
     \item[1.] for every $v\in \scrG_{\Z}$ and $\alpha\in \Phi_i$, either $\norm{\pi^{\rmS_i}_{\alpha}(v)}=0$ or $\norm{\pi^{\rmS_i}_{\alpha}(v)}\geq 1/C_5$;
    \item[2.] for every $v\in \scrG$ and $\alpha\in \Phi_i$, $\norm{v}\geq 1/C_5 \norm{\pi^{\rmS_i}_{\alpha}(v)}$;
    \item[3.] for every $v \in \scrG$ and $\alpha\in\Phi_i^{0+}$, $\norm{\pi^{\rmS_i}_{0+}(v)} \geq {1}/{C_5} \norm{\pi^{\rmS_i}_{\alpha}(v)}$ .
\end{itemize}

As a result of Proposition \ref{propDGUL} we obtain the following:

\begin{coro}\label{coroDGUL}
There exist $\eta_1>0$ and a function $\ep_1: (0,\eta_1)\to(0,\infty)$  with $\lim_{\eta\to0} \ep_1(\eta)=0$ such that for all $0<\eta<\eta_1$ and $g\in \rmG$ such that $[g]\notin X^{\rmM}_{\eta}$, there exists $m\in \rmM$ and an $\R$-parabolic subgroup $\bmP$ of $\bmG$ containing $\bmM$ such that $g^{-1}\bmP g$ is defined over $\Q$ and if we let $\scrU$ be the Lie algebra of $\rmU$, which is the real points of the unipotent radical of $\bmP$, then
\begin{enumerate}
    \item 
    $\norm{\scrU\cap \scrG_g}^{\frac{1}{\dim\scrU}} <\ep_1(\eta)$;
    \item for all $v\in \scrG_{mg}\setminus \scrU$, the orthogonal projection of $v$ to the orthogonal complement of $\scrU$ satisfies $\norm{\pi_{\scrU^{\perp}}(v)}\geq \left(\frac{1}{C_4C_5}\right)^2$;
    \item for any $\scrG_{g}$-rational, $\rmM$-stable subspace $\calL$ that is not contained in $\scrU$, we have $\norm{\calL\cap\scrG_g}^{\frac{1}{\dim\calL}} \geq 2C_8 \delta_{\rmM}([g])$ with $C_8>1$ as in Lemma \ref{lemExpand_general} below.
\end{enumerate}
\end{coro}

As the reader will see, $C_8$ could be replaced by any positive constant except that one needs to modify $\eta_1$ accordingly.

\begin{proof}
Take $\eta\in(0,\eta_0)$ and $g\notin X^{\rmM}_{\eta}$. By Proposition \ref{propDGUL}, find $i=i_g$ such that
\begin{equation*}
    g\in \Sigma^{\rmM}_{i,B_0\cap \rmP_i,\theta_0,\ep_0(\eta)}\Gamma.
\end{equation*}
By unwrapping the definition, there exists $\gamma_g\in\Gamma$ such that $\gamma_g^{-1}g^{-1}\rmM g\gamma_g $ is contained in ${}^{\circ}\!\rmP_i$ (as $\bmM$ is semisimple and $\rmM$ is connected, any conjugate of $\rmM$ being contained in $\bmP_i$ automatically implies being contained in ${}^{\circ}\!\rmP_i$) and if $g\gamma_g=k_ga_gp_g$ is the horospherical coordinate of $g\gamma_g$ with respect to $\rmP_i$ then 
\begin{enumerate}
    \item $\alpha(a_g)<\theta_0,\,\forall \alpha\in\Delta_i$;
    \item $\alpha_g(a_g)<\ep_0(\eta), \exists \alpha_g \in \Delta_i$;
    \item $p_g\gamma_g^{-1}g^{-1}m_gg\gamma_g=b_g\lambda_g$ for some $m_g\in \rmM$, $b_g\in B_0\cap {}^{\circ}\!\rmP_i$ and $\lambda_g\in {}^{\circ}\!\rmP_i\cap \Gamma$.
\end{enumerate}
Recall that we have chosen $\theta_0$ such that $\alpha(a_g)<1$ for all $\alpha\in\Phi_i^{-}$. Hence $\alpha(a_g)\geq 1$ for all $\alpha\in\Phi_i^{+0}$.

Now take $m:=m_g$, $0<\eta_1<\eta_0$, which will be determined later at Equation \ref{equa_defi_eta_1} in the proof of (3). Let $\bmP:=g\gamma_g\bmP_i\gamma_g^{-1}g^{-1}$. Then $\bmP$ contains $\bmM$ and $g^{-1}\bmP g= \gamma_g\bmP_i\gamma_g^{-1}$ is defined over $\Q$. 
Also let $\bmU$ be the unipotent radical $\bmP$.
Define $\ep_1(\eta):=\ep_0(\eta)^{1/N}C_3^{1/N}$.
It remains to prove the three claims.

\subsubsection*{Proof of (1)}
Recall that $N$ denotes the dimension of $\bmG$.

\[
\begin{aligned}
\norm{\scrU\cap \scrG_g} &=\norm{(g\gamma_g\cdot \scrU_i)\cap (g\gamma_g\cdot \scrG_{\Z})}\\
&=\norm{ (g\gamma_g)\cdot (\scrU_i \cap \scrG_{\Z}) }\\
&= \norm{ (k_ga_gp_g) \cdot (\scrU_i \cap \scrG_{\Z}) }.
\end{aligned}
\]
First note that $p_g$ preserves $\scrU_i$ and preserves the (co)volume. On the other hand $a_g$ also preserves $\scrU_i$ but $\alpha(a_g)<1$ for all $\alpha$ appearing in $\scrU_i$ and for $\alpha=\alpha_g$, which appears in $\scrU_i$, $\alpha(a_g)<\ep_0(\eta)$. So $|\det(\Ad(a_g)\vert_{\scrU_i})|<\ep_0(\eta)$. Hence
\[
\begin{aligned}
&\norm{\scrU\cap \scrG_g} \leq \ep_0(\eta)\norm{\scrU_{i}\cap\scrG_{\Z}}\leq \ep_0(\eta)C_3.\\
\implies &
\norm{\scrU\cap \scrG_g}^{\frac{1}{\dim\scrU}} \leq 
\left(
\ep_0(\eta)C_3
\right)^{1/N}
=\ep_1(\eta).
\end{aligned}
\]
This proves (1).
Note this also shows that $\delta_{\rmM}([g])\leq \ep_1(\eta)$.

\subsubsection*{Proof of (2)}
Take $v\in \scrG_{mg}\setminus \scrU$. 
As $\scrG_{mg}=mg\cdot \scrG_{\Z}=mg\gamma_g\lambda_g^{-1}\cdot\scrG_{\Z}$ 
and $\scrU= mg\gamma_g\lambda_g^{-1}\cdot\scrU_i$, we can find $v_g\in\scrG_{\Z}\setminus \scrU_i$ such that 
\[ v=mg\gamma_g \lambda_g^{-1}\cdot v_g. \]
Hence 
\[
\begin{aligned}
v &= ( g\gamma_{g}) \cdot (\gamma_g^{-1}g^{-1}mg\gamma_g)\cdot \lambda^{-1}_g \cdot v_g \\
&= (k_ga_g)\cdot  (p_g\gamma_g^{-1}g^{-1}mg\gamma_g)\cdot \lambda^{-1}_g \cdot v_g \\
&= k_ga_gb_g\cdot v_g.
\end{aligned}
\]

Note that $\scrU=g\gamma_g\cdot \scrU_{i}=k_ga_gp_g\cdot \scrU_i =k_g\cdot \scrU_i$ and $\Ad(k_g)$ acts by isometry, we have 
\[
\Ad(k_g)\circ \pi_{\scrU_i^{\perp}} = \pi_{\scrU^{\perp}}\circ \Ad(k_g).
\]
Also recall that $\pi_{\scrU_i^{\perp}}= \pi^{\rmA_i}_{0+}$ and
\[
\Ad(w_i)\circ \pi^{\rmA_i}_{0+} = \pi^{\rmS_i}_{0+}\circ \Ad(w_i).
\]
Now if we write $s_g=w_ia_gw_i^{-1}\in \rmS_i$ and $b_g=h^i_gu^i_g$ for some $h^i_g\in Z_{\rmG}(\rmA_i)$ and $u^i_g\in \rmU_i$, then
\[
\begin{aligned}
\norm{\pi_{\scrU^{\perp}}(v)} &= 
\norm{ \pi_{\scrU^{\perp}}(k_ga_gb_g\cdot v_g)} = \norm{ k_g\cdot \pi_{\scrU_i^{\perp}}(a_gb_g\cdot v_g)}
= \norm{ 
\pi^{\rmA_i}_{0+}(a_gb_g\cdot v_g)}\\
&= \norm{ 
w_i^{-1} \cdot \pi^{\rmS_i}_{0+}((w_ia_gw_i^{-1})(w_ib_g)\cdot v_g)
}\\
& \geq \frac{1}{C_4} \norm{   \pi^{\rmS_i}_{0+}((s_g) (w_ih^i_{g}w_i^{-1}) (w_iu^i_g)\cdot v_g)}
\end{aligned}
\]
Let $\alpha_g\in \Phi^{0+}_i$ be 
a maximal element (our convention about the partial order is that $\alpha-\beta$ is contained in positive combinations of $\Delta_i$ iff $\alpha\leq \beta$) 
such that $\pi^{\rmS_i}_{\alpha_g}(v_g)\neq 0$.  
 Then for any $u\in \rmU_i$, $\pi^{\rmS_i}_{\alpha_g}(u\cdot v_g)=\pi^{\rmS_i}_{\alpha_g}(v_g)$.
Also  the reader is reminded that $\norm{\pi^{\rmS_i}_{\alpha_g}(v_g)}\geq 1/C_5$ and $C_4$ bounds the operator norm of some elements. Recall that $\alpha_g(s_g)>1$ as we have chosen $\theta_0$ small enough.

We may continue the above inequalities as 
\[
\begin{aligned}
\norm{\pi_{\scrU^{\perp}}(v)} &\geq 
\frac{1}{C_4C_5} 
 \norm{   \pi^{\rmS_i}_{\alpha_g}((s_g) (w_ih^i_{g}w_i^{-1}) (w_iu^i_g)\cdot v_g)}\\
 &\geq 
 \frac{1}{C_4C_5} 
  \norm{ (w_ih^i_{g}w_i^{-1})\cdot  \pi^{\rmS_i}_{\alpha_g}(  w_iu^i_g\cdot v_g)}\\
  &\geq \frac{1}{C_4C_5C_4}  \norm{  \pi^{\rmS_i}_{\alpha_g}(  v_g)}
  \geq \left(\frac{1}{C_4C_5}\right)^{2}.
\end{aligned}
\]

\subsubsection*{Proof of (3)}
We are going to use both (1) and (2) here.
As $\calL$ is $\rmM$-stable, $\calL$ is also $\scrG_{mg}$-rational and $\norm{\calL\cap{\scrG_{g}}}=\norm{\calL\cap \scrG_{mg}}$. This is the only place we need $\calL$ to be $\rmM$-stable.

As $\scrU$ is also $\scrG_{mg}$-rational, we have that $\pi_{\scrU^{\perp}}(\calL\cap\scrG_{mg})$ is a lattice in $\scrU^{\perp}$ (we are not claiming that $\pi_{\scrU^{\perp}}$ is $\scrG_{mg}$-rational, but $\pi_{\scrU^{\perp}}$ is a lift of the map of quotient by $\scrU$, which is $\scrG_{mg}$-rational) and
\[
\norm{\calL\cap\scrG_{mg}} = \norm{\calL\cap \scrU \cap\scrG_{mg}}
\cdot \norm{\pi_{\scrU^{\perp}}(\calL\cap\scrG_{mg})}.
\]
By (2), $\pi_{\scrU^{\perp}}(\calL\cap\scrG_{mg})$ has its shortest non-zero vector with length at least  $1/(C_4C_5)^2$. Therefore it has a fundamental domain containing a cube of size $1/2N(C_4C_5)^2$. Hence 
\[
\norm{\pi_{\scrU^{\perp}}(\calL\cap\scrG_{mg})} \geq 
\left(
\frac{1}{2NC_4^2C_5^2}
\right)^N.
\]
For simplicity we let
\[
\begin{aligned}
&C_6:= {(2NC_4^2C_5^2)}^N, \\
&C_7^{-1}:= \min \left\{\left|\frac{1}{x}-\frac{1}{y}\right|
\:\middle\vert \:
x\neq y, \, x,y\in\{1,...,N\}
\right\}.
\end{aligned}
\]
And we choose $\eta_1$ small enough such that for any $0<\eta<\eta_1$,
\begin{equation}\label{equa_defi_eta_1}
    \begin{aligned}
&2^NC_8^NC_6^N \ep_1(\eta) <1,\\
&C_6^{-1} 
\left(
2^NC_8^NC_6^N\ep_1(\eta)
\right)^{-C_7^{-1}} \geq 2C_8.
\end{aligned}
\end{equation}
where $C_8$ is as in Lemma \ref{lemExpand_general}.

Now there are two cases depending on how large $\norm{\calL\cap \scrU \cap\scrG_{mg}}$ is. 

\textit{Case I}, 
$\norm{\calL\cap \scrU \cap\scrG_{mg}}^{1/\dim (\calL \cap \scrU)} \leq (2C_8C_6)^N\delta_{\rmM}[g]$.

Note by our assumption $\dim \calL > \dim (\calL\cap\scrU)$ and hence $(\dim \calL)^{-1} - (\dim \calL\cap \scrU)^{-1}<0$.
Also, $\delta_{\rmM}([g])\leq \ep_1(\eta)$ by part (1).
\[
\begin{aligned}
\norm{\calL\cap {\scrG_{mg}}}^{\frac{1}{\dim \calL}}
&\geq
C_6^{-1/\dim \calL}
\cdot 
\norm{\calL\cap \scrU\cap {\scrG_{mg}}}^{\frac{1}{\dim \calL}-\frac{1}{\dim \calL\cap \scrU}} \cdot
\norm{\calL\cap \scrU\cap {\scrG_{mg}}}^{\frac{1}{\dim \calL\cap \scrU}} \\
&\geq 
C_6^{-1} \cdot 
\left(
2^NC_8^NC_6^N\ep_1(\eta)
\right) ^{\dim(\calL\cap\scrU)\cdot ({\frac{1}{\dim \calL}-\frac{1}{\dim \calL\cap \scrU}})}
\cdot \delta_{\rmM}([g]) 
\end{aligned}
\]
By assumption 
$0<2^NC_8^NC_6^N\ep_1(\eta)<1$ and 
${\dim(\calL\cap\scrU)\cdot ({\frac{1}{\dim \calL}-\frac{1}{\dim \calL\cap \scrU}})}\leq -C_7^{-1}<0$, hence we may continue:
\[
\begin{aligned}
\norm{\calL\cap \scrG_{mg}}^{\frac{1}{\dim \calL}}
&\geq  C_6^{-1} \cdot
\left(
2^NC_8^NC_6^N\ep_1(\eta)
\right)^{-C_7^{-1}}\cdot \delta_{\rmM}([g]) \geq 2C_8\delta_{\rmM}([g]) ,
\end{aligned}
\]
which completes the case I.

\textit{Case II},  
$\norm{\calL\cap \scrU\cap \scrG_{mg} }^{1/\dim (\calL \cap \scrU)}\geq (2C_8C_6)^N\delta_{\rmM}[g]$.

This case is more direct.
\[
\begin{aligned}
\norm{\calL\cap \scrG_{mg}}^{\frac{1}{\dim \calL}} &\geq C_6^{-1/\dim\calL} 
\left(2^NC_8^NC_6^N\delta_{\rmM}([g])
\right)^{\frac{\dim (\calL\cap\scrU)}{\dim \calL}}\\
&\geq
C_6^{-1}2C_8C_6\delta_{\rmM}([g])=2C_8\delta_{\rmM}([g]).
\end{aligned}
\]
Now the proof of the corollary is complete.
\end{proof}

\begin{lem}\label{lemExpand_general}
There exist two constants $C_8, C_9>1$ and a finite set $\calF\subset S$ such that for any $\R$-parabolic subgroup $\bmP$ of $\bmG$ containing $\rmM$, there exists $s\in \calF$ such that
\begin{enumerate}
    \item $\displaystyle\norm{sv}>\frac{1}{C_8}\norm{v}$ for all pure wedges $v$ in $\scrG$;
    \item $\norm{sv}>C_9 \norm{v}$ for all pure wedges $v$ with $\calL_v \subset \scrU$.
\end{enumerate}
\end{lem}

\begin{proof}
The proof is similar to that of Lemma \ref{lemExpand_SL_N} where the key is to produce a bijective projection onto some `coordinate plane' $V_I$. Here we will produce a bijective projection to some `coordinate horospherical Lie subalgebra'.

Let $\calP_M$ be the collection of $\R$-parabolic subgroups of $\bmG$ containing $\bmM$. 
For each $\bma_t \in X_*(\bmS)$, we let 
\[
\bmP_{\bma_t}:=\left\{
g\in \bmG \:\middle\vert \:
\lim_{t\to0} \bma_tg\bma_t^{-1} \text{ exists }
\right\}.
\]
This way we get a finite collection of elements in $\calP_M$. Label this set as $\{\bmP_i\}_{i\in\calB}$ for some finite set $\calB$. As $\bmS\leq \bmZ_{\bmG}\bmM$ is a maximal $\R$-split torus, for any other $\bmP\in \calP_M$, there exists $h\in Z_{\rmG}(\rmM)$ such that $\bmP=h\bmP_ih^{-1}$ for some $i\in \calB$. Hence $\calP_M$ is a finite union of compact homogeneous spaces of $Z_{\rmG}(\rmM)$. Indeed, the stabilizer of each $\bmP_i$ in $\bmZ_{\bmG}(\bmM)$ is a parabolic subgroup of $\bmZ_{\bmG}(\bmM)$.

Now fix $\bmP\in \calP_M$ and $h\in  Z_{\rmG}(\rmM)$ with $\bmP=h\bmP_i h^{-1}$. By Bruhat decomposition (see \cite[21.15]{BorLinearAG}), there exists $w\in N_{Z_{\rmG}(\rmM)}\rmS$, $u\in \rmU$ and $p\in \rmP_i$ such that $h=uwp$ where $\bmU$ is a maximal $\R$-unipotent subgroup contained in $\bmP_i$ that is normalized by $\bmS$. Then $\bmP= h\cdot \bmP_i = uwp\cdot \bmP_{i}= u\cdot \bmP_{j_0}$ for some $j_0\in \calB$, where $\cdot$ denotes the conjugation by a group element.

We claim that the $\rmS$-equivariant projection $\pi_{j_0}: \scrG \to \scrU_{j_0}$, when restricted to $\scrU$, is bijective (but usually $\scrU$ is not $\rmS$-stable). 

Decompose $\scrG=\bigoplus_{\alpha\in\Phi} \scrG^{\rmS}_{\alpha}$ with respect the $\Ad(\rmS)$-action. And let $\pi^{\rmS}_{\alpha}$ be the associated projection onto $\scrG_{\alpha}^{\rmS}$. 
Fix a minimal $\R$-parabolic $\bmP_{min}$ with $\bmS\cdot\bmU \leq\bmP_{min}\leq\bmP_i$.
A partial order $\leq_i$ is defined on $\Phi$ by demanding that $\alpha\leq_i \beta$ iff $\beta - \alpha$ is a weight that appears in $\scrP_{min}$, the Lie algebra of $\rmP_{min}$.

Now take $v\in \scrU$, so $v=u\cdot w$ for some $w\in \scrU_{j_0}$. Take $\alpha_0\in\Phi$ to be a minimum element such that $\pi^{\rmS}_{\alpha_0}(w)\neq 0$. Then 
\[
\pi^{\rmS}_{\alpha_0}(u\cdot w) = \pi^{\rmS}_{\alpha_0}(u\cdot \pi^{\rmS}_{\alpha_0}(w)) = \pi^{\rmS}_{\alpha_0}( w) \neq 0.
\]
As $\scrU_{j_0}$ is defined by a cocharacter of $\bmS$ and projects to $\scrG^{\rmS}_{\alpha_0}$ nontrivially,  $\scrU_{j_0}$ necessarily contains $\scrG^{\rmS}_{\alpha_0}$.
Thus $\pi_{j_0}(v)=\pi_{j_0}(u\cdot w)\neq 0$. This proves that $\pi_{j_0}\vert_{\scrU}: \scrU \to \scrU_{j_0}$ is injective. And hence it is bijective because $\scrU$ and $\scrU_{j_0}$  have the same dimension.

The rest of the proof follows the same lines as in Lemma \ref{lemExpand_SL_N} and is omitted.
\end{proof}

\begin{prop}
There exists $\eta_2>0$ and $C_{10}>1$ such that for all $[g]\notin X^{\rmM}_{\eta_2}$, there exists $s \in \rmS$ such that $\delta_{\rmM}([sg])>C_{10}\delta_{\rmM}([g])$. Consequently, there exists $s\in \rmS$ such that $[sg]\in X^{\rmM}_{\eta_2}$.
\end{prop}

Theorem \ref{thmMain} would follow from this proposition by  \cite[Theorem 4.6]{DawGoroUll18} just as in last section.

\begin{proof}
We prove the proposition with $\eta_2:=\eta_1$ and $C_{10}:= \min\{C_9^{1/N},2\}$. 

Take any $[g]\notin X^{\rmM}_{\eta_2}$.

Find $\bmP$ an $\R$-parabolic subgroup of $\bmG$ according to Corollary \ref{coroDGUL}. Then choose $s\in \rmS$ with Lemma \ref{lemExpand_general} being applied to $\bmP$.
It suffices to prove that $\delta_{\rmM}([sg])\geq C_{10}\delta_{\rmM}([g])$.

Take an $\R$-subspace $\calL'$ of $\scrG$ that is $\scrG_{sg}$-rational and $\rmM$-stable. 
Then $\calL:=s^{-1}\calL'$ is $\scrG_g$-rational and $\rmM$-stable. Also, $\calL'\cap\scrG_{sg}=s \cdot (\calL\cap \scrG_{g})$. There are two cases.

If $\calL \subset \scrU$, then
\[
\norm{\calL'\cap\scrG_{sg}}^{\frac{1}{\dim \calL'}}
\geq C_9^{\frac{1}{\dim \calL'}}
\norm{\calL \cap \scrG_{g}}^{\frac{1}{\dim \calL}}
\geq  C_9^{\frac{1}{\dim \calL'}}\delta_{\rmM}([g])
\geq  C_9^{\frac{1}{N}}\delta_{\rmM}([g])
.
\]

If $\calL \nsubseteq \scrU$, then
\[
\norm{\calL'\cap \scrG_{sg}}^{\frac{1}{\dim \calL'}}
\geq C_8^{-\frac{1}{\dim \calL'}}
\norm{\calL\cap \scrG_{g}}^{\frac{1}{\dim \calL}}
\geq  C_8^{-\frac{1}{\dim \calL'}} 2C_8\delta_{\rmM}([g]) \geq 2\delta_{\rmM}([g]).
\]

Hence our proof is complete.

\end{proof}

\section{Proof of Theorem \ref{thmEpi} and \ref{thmUniform}}
\label{secProofThmGeneralized}

\begin{proof}[Proof of Theorem \ref{thmEpi}]
Indeed, by  \cite[Corollary 1.3]{ShahWeiss00}, for every $x\in \rmG/\Gamma$, $\overline{\rmF x}$ is $\rmH$-invariant, thus has to intersect some bounded set independent of $x$ nontrivially.
\end{proof}

\begin{proof}[Proof of Theorem \ref{thmUniform}]
First we claim that up to $\rmG$-conjugacy, there are only finitely many elements from $\calH$. 
For every dimension, up to isomorphism, there are only finitely many real semisimple algebras (see \cite{Knap02}). And for each of them, the Lie subalgebras of $\bmG$ that are isomorphic to this one form a finite orbit under $\rmG$ by \cite[Lemma A.1]{EinMarVen09} (see also \cite{Richard67}). Now we enumerate the corresponding real algebraic subgroups as $\{\bmM_1,...,\bmM_s\}$. 
Also fix a maximal $\R$-split torus $\bmS_i$ of $\bmZ_{\bmG}\bmM_i$.
By shortening the list, we assume that each $\bmH_i:=\bmM_i \cdot \bmS_i$ belongs to $\calH$.
We now argue that each $\bmH\in \calH$ is isomorphic to one of $\bmH_i$. Indeed, write $\bmH$ as an almost direct product $\bmM\cdot \bmS$. Then there exists $g \in \rmG$ such that $g \bmM g^{-1}=\bmM_i $ for some $i$. Then $g \bmS g^{-1}$ is a maximal $\R$-split torus of $\bmZ_{\bmG}\bmM_i$, thus there is some $h\in \bmZ_{\rmG}\rmM_i$ such that $hg \bmS g^{-1} h^{-1}=\bmS_i$ (see \cite[15.2.6]{Spr98}). Thus $hg \bmH g^{-1}h^{-1}= \bmH_i$.

As this is a finite list, we find a compact set $C$ of $\rmG/\Gamma$ such that for every $i$ with $\bmH_i\in\calH$ and every $x\in \rmG/\Gamma$, there exists $h\in\rmH_i$ with $hx \in C$. For each $i$, find a compact subgroup $K_i$ such that $Z_{\rmG}\rmH_i = K_i\cdot \rmS_i$. 
Let $C'$ be a larger compact subset containing $K_i \cdot C$ for all $i$'s. 
We fix some embedding of $\bmG \to \SL_N$ that induces a proper map $\rmG/\Gamma \to \SL_N(\R)/\SL_N(\Z)$. 
For each $i$, fix a nonempty bounded open set $\Omega_i$ in $\rmH_i$.

Thus by \cite[Theorem 1.1]{RicSha18}, for each $i$, there exists a closed subset $Y_i$ of $\rmG$ such that 
\begin{itemize}
    \item[1.] $\rmG= Y_i \cdot K_i \cdot \rmS_i$;
    \item[2.]  there exists some $c>0$ such that 
    \begin{equation*}
        \sup_{\omega \in \Omega_i} \norm{
        y\omega \cdot v 
        }\geq c\norm{v}\quad \forall i, \,y\in Y, \,v\in \R^N.
    \end{equation*}
\end{itemize}
Through the work of \cite{EskMozSha97} (or \cite{KleMar98}), this implies that there exists a compact set $C''\subset \rmG/\Gamma$ such that  for every $x\in C'$ and every $y \in Y_i$, 
\begin{equation*}
    y\Omega_i \cdot x  \cap C'' \neq \emptyset.
\end{equation*}
 By further using $\rmG= Y_i \cdot K_i \cdot \rmS_i$ and $K_i\cdot  C \subset C'$, we have that for every $x\in C$ and every $g\in \rmG$,
 \begin{equation*}
     g\rmH_i \cdot  x \cap C'' \neq \emptyset.
 \end{equation*}

Now take $x_0\in \rmG/\Gamma$ and $\bmH \in\calH$, we wish to show $\rmH \cdot x_0 \cap C''\neq \emptyset$.
First find $g_H\in \rmG$ such that $\bmH=g_H \bmH_i g_H^{-1}$ for some $i$. Then by Theorem \ref{thmMain},
\begin{equation*}
    \rmH\cdot x_0 = g_H \rmH_i g_H^{-1}\cdot x_0 =g_H \bmH_i \cdot x_0'\quad \exists \,x_0'\in C,
\end{equation*}
which intersect with $C''$ nontrivially. Hence we are done.

\end{proof}

\section{Geometric consequences}\label{secGeometric}

Main results in this section is a proof of Theorem \ref{thmGeometric} and a characterization of rigid totally geodesic submanifolds in Proposition \ref{propRigidChar}.
For backgrounds on symmetric spaces, our main references here are \cite{Helga01} and \cite{KozMau18}.

\subsection{Arithmetic quotients of symmetric spaces of noncompact type}
Let $\wtX$ be a connected global Riemannian symmetric space of noncompact type. Thus the identity connected component of isometry group of $\wtX$ is a connected semisimple Lie group $G$ and $\wtX$ is identified with the space of maximal compact subgroups of $G$. By fixing a maximal compact subgroup $K_0$ of $G$ , $X$ is identified with $K_0\bs G$. We are going to assume $G=\rmG:=\bmG(\R)^{\circ}$ for a connected $\Q$-algebraic group $\bmG$ and take $\Gamma \leq \rmG\cap \bmG(\Q)$ to be an arithmetic lattice. Call a locally symmetric space $X$ of the form $X= \wtX/\Gamma= K_0\bs \rmG / \Gamma$ arithmetic. We assume $\Gamma$ to be neat so that $X$ is a Riemannian manifold. 

For a maximal compact subgroup $K$ of $\rmG$, there exists a unique algebraic Cartan involution $\iota_{K}$ over $\R$ associated with $K$. 
If $\bmK$ is the fixed point of $\iota_{K}$ in $\bmG_{\R}$, then $K=\rmK:=\bmK(\R)^{\circ}$. 
By abuse of notation we also use $\iota_{K}$ to denote its induced action on the Lie algebra. Thus the Lie algebra $\frakk$ of $K$ is identified with those fixed by $\iota_{K}$. 
Let $\frakp_{\rmK}$ be the $(-1)$-eigenspace of $\iota_{K}$ in $\frakg$, the Lie algebra of $\rmG$. Then
\begin{equation*}
    \frakg= \frakk \oplus \frakp_{\rmK}.
\end{equation*}
For simplicity we write $\iota_0:= \iota_{\rmK_0}$ and $\frakp_0:=\frakp_{\rmK_0}$.

Let 
\begin{equation*}
    B(v,w):= - \Tr (\ad(v)\ad(\iota_{0} w))
\end{equation*}
be the positive definite bilinear form on $\frakg$, identified with the tangent space of $\rmG$ at $id$, associated with $\rmK_0$. By right translation, we get a right $\rmG$-invariant Riemannian metric on $\rmG$. This metric is also left $\rmK_0$-invariant.
This metric thus induces metrics on $\rmG/\Gamma$, $\rmK_0\bs \rmG/\Gamma$ and their closed submanifolds.
All the ``$\Vol$'' appearing below will be referred to measures induced from this metric. Up to scalars (to be more precise, for each irreducible factor there is a positive scalar), the original Riemannian metric on $X$ coincides this one.

\subsection{Totally geodesic submanifolds}

A totally geodesic submanifold $\wtY$ of $\wtX$ is again a symmetric space. By \cite[Theorem 7.2]{Helga01}, there exists a triple system $\fraks_0 \subset \frakp_{0}$, i.e. $[x,[y,z]]\in \fraks_0$ if $x,y,z \in \fraks_0$ and $g\in \rmG$, such that 
$\wtY= \rmK_0\bs \rmK_0 \exp(\fraks_0)g $. Then $\frakh_Y:= \fraks_0 + [\fraks_0 ,\fraks_0 ]$  is a $\iota_0$-stable Lie subalgebra. By writing $H_0$ for the corresponding Lie subgroup, we have $\wtY= \rmK_0\bs \rmK_0  H_{0} g $. The choice of $\fraks_0$ is not unique but it is understood when we write $\wtY= \rmK_0\bs \rmK_0  H_{0} g $.

Now let $Y \subset X$ be an embedded totally geodesic submanifold, by choosing a lift of some point of $y\in Y$, we have a unique totally geodesic submanifold $\wtY$ of $\wtX$ who projects to $Y$ as a local isometry. Thus $Y= \rmK_0\bs \rmK_0  H_{0} g\Gamma/\Gamma$. It is not hard to verify that $H_0 g \Gamma/\Gamma$ is also closed in $\rmG/\Gamma$.

We are going to be interested in finite-volume embedded totally geodesic submanifold $Y$ of $X$. Using the notation as in the last paragraph, one can verify from the definition that for such a $Y$,
\begin{equation}\label{equaVolumeRelation}
    \Vol(H_0 g \Gamma/\Gamma) = \Vol(Y) \cdot \Vol(\rmK_0\cap  H_{0}).
\end{equation}

\subsection{Rigid totally geodesic submanifolds}

\begin{defi}
Fix a natural number $N$, let 
\begin{equation*}
    \calT\calG^N:=
\left\{
Y \subset X\;\middle\vert\;
\parbox{70mm} 
    {
    \raggedright   $Y$ {is a embedded totally geodesic submanifold of }$X$, $\Vol(Y)<\infty$, $\dim Y=N$
    }
\right\}
\end{equation*}
be equipped with the Chabauty topology (see \cite[E.1]{BenePetr92}). We say that $Y \in \calT\calG^N$ is \textbf{rigid} if $\{Y\}$ is open in $\calT\calG^N$. The collection of such $Y$'s are denoted as $\calT\calG^{N,\rigid}$.
\end{defi}

Take $Y := \rmK_0 \bs \rmK_0  H_{0}g \Gamma/\Gamma \in \calT\calG^{N} $. 
We can write 
$\frakz_{\frakg}(\frakh_Y)= \frakk_{Y} \oplus \frakz(\frakh_Y) $ for some $\iota_0$-stable subalgebra $\frakk_{Y}$ centralizing $\frakz(\frakh_Y)$
where $\frakz_{\frakg}(\frakh_Y)$ is the centralizer of $\frakh_Y$ in $\frakg$ and $\frakz(\frakh_Y)$ is the center of $\frakh_Y$.  If $Y$ is rigid, then $\frakk_{Y}$ is contained in $\frakk_0$ for otherwise there exists $v_{\neq 0}\in \frakk_{Y}\cap \frakp_0$ and 
\begin{equation*}
    Y_t:= \rmK_0 \bs \rmK_0 \exp(vt)  H_{0}g \Gamma/\Gamma \to Y 
    \quad\text{as } t \to 0
\end{equation*}
in $\calT\calG^N$ and $Y_t \neq Y$ for $t\neq 0$.
The converse is also true.

\begin{prop}\label{propRigidChar}
Notations as above. Then $Y\in \calT\calG^{N}$ is rigid iff  $\frakk_{Y}$ is contained in $\frakk_0$. In this case $\frakh_Y$ is algebraic.
\end{prop}

The second claim follows from the following lemma. The proof of the rest of the claim is delayed to the next section.
Recall that a Lie subalgebra of the Lie algebra of a linear algebraic group over $\R$ is said to be algebraic iff its the Lie algebra of some algebraic subgroup over $\R$ (see \cite[Chapter II.7]{BorLinearAG}).

\begin{lem}\label{lemmaRigidisAlgebraic}
Let $\frakh$ be a $\iota_0$-stable Lie subalgebra of $\frakg$ with no compact factors. Assume that all noncompact factors of $\frakz_{\frakg}\frakh$ are contained in $\frakh$, then 
\begin{itemize}
    \item[1.] $\frakh$ is algebraic;
    \item[2.] identity coset is the unique $x\in \rmK_0\bs \rmG$ such that $x \rmH$ is totally geodesic;
    \item[3.] there exists a finite list (only depends on $\rmG$, $\rmK_0$) $\{\frakh_1,...\frakh_k\}$ satisfying the same condition as $\frakh$ does such that $\frakh$ is conjugate to one of them via $\rmK_0$.
\end{itemize}
\end{lem}

\begin{proof}
Write $\frakh=\fraka\oplus \frakm$ for some abelian subalgebra $\fraka$ in $\frakp_0$ and semisimple Lie subalgebra $\frakm$. Moreover $\fraka$ commutes with $\frakm$. By \cite[II.7.9]{BorLinearAG}, $\frakm$ is algebraic. 
As $\frakz_{\frakg}\frakh$ is algebraic and $\fraka$ is characterized as the ($-1$)-eigenspace of $\iota_0$, which is algebraic, in $\frakz_{\frakg}\frakh$, we have $\fraka $ is algebraic. Hence $\frakh$ is algebraic.

Item 2. and 3. have been proved in \cite[Section 2]{KozMau18} under the additional assumption that $\frakh$ is semisimple. But the same proof presented there also works without this assumption.
\end{proof}

Item 2. and 3., together with Equa. \ref{equaVolumeRelation} imply the following
\begin{lem}\label{lemmaVolBounded}
There exists a constant $C_{11}>1$ such that
\begin{equation*}
   C_{11}^{-1} \Vol (Y)\leq \Vol( H_0g\Gamma/\Gamma ) \leq C_{11} \Vol (Y)
\end{equation*}
for all $Y$ in $\calT\calG^{N,\rigid}$.
\end{lem}

\subsection{Proof of Proposition \ref{propRigidChar}}
It remains to show that, assuming $\frakk_{Y}$ is contained in $\frakk_0$, for a sequence of 
$Y_i=  \rmK_0 \bs \rmK_0 H_{Y_i} g_i \Gamma/\Gamma $ converging to $Y=  \rmK_0 \bs \rmK_0 \rmH_{0} g_0 \Gamma/\Gamma$, then $Y_i=Y$ for $i$ sufficiently large.
Replacing $\Gamma$ by $g_0 \Gamma g_0^{-1}$, assume $g_0=id$. 
Also write $\frakh_0 = ( \frakk_{\frakh_0}\cap \frakh_0) \oplus \fraks_0$ where $\fraks_0$ is a triple system associated with $Y$. Also write $\rmH_0= \rmA_0 \cdot \rmM_0$ (at the level of Lie algebra, $\frakh_0=\fraka_0\oplus\frakm_0$) as an almost product between its center and the semisimple part.

For an element $g\in \rmG$, let ${}_{\rmK_0}[g]_{\Gamma}$ (resp. $[g]_{\Gamma}$, ${}_{\rmK_0}[g]$) denote its image in $\rmK_0 \bs \rmG/\Gamma$ (resp. $\rmG/\Gamma$, $\rmK_0 \bs \rmG$).

From definition,  
\begin{equation*}
    k_i h_{Y_i} g_i \gamma_i =\ep_i
\end{equation*}
for some sequences of $k_i \in \rmK_0$, $h_{Y_i}\in H_{Y_i}$, $\gamma_i \in \Gamma$ and $\ep_i \in \rmG$ with $\ep_i \to id$.
Thus 
\begin{equation*}
    Y_i=  \rmK_0 \bs \rmK_0 H_{i} \ep_i \Gamma/\Gamma,
    \quad \text{with }H_i:= k_i H_{Y_i}k_i^{-1}.
\end{equation*}
Let $\fraks_i:= \Ad(k_i)\cdot \fraks_{Y_i}$ (recall $\fraks_{Y_i}\subset \frakp_0$ is the triple system associated with $Y_i$). For $\delta>0$, let $B_{\frakp_0,\delta}$ be the open ball of radius $\delta$ in $\frakp_0$. We choose $\delta>0$ small enough such that 
\begin{equation*}
    B_{\frakp_0,\delta} \to \rmK_0\bs \rmG/\Gamma:\;\; 
    v \mapsto {}_{\rmK_0}[\exp(v)]_{\Gamma}
\end{equation*}
is a homeomorphism onto its image.  By passing to a subsequence assume $\fraks_i$ converges to $\fraks_{\infty}$. Then $\fraks_{\infty}$ is still a triple system and hence $\frakh_{\infty}:=\fraks_{\infty}\oplus [\fraks_{\infty},\fraks_{\infty}]$ is a $\iota_0$-stable Lie subalgebra. 
Also ${}_{\rmK_0}[\exp(\fraks_i \cap B_{\frakp_0,\delta})]_{\Gamma}$  converges to ${}_{\rmK_0}[\exp(\fraks_{\infty} \cap B_{\frakp_0,\delta})]_{\Gamma}$, which must be contained in $\fraks_0$ by assumption.  Since they share the same dimension we conclude that $\frakh_{\infty}=\frakh_{0}$. Thus $\frakh_i$'s are all algebraic by Lemma \ref{lemmaRigidisAlgebraic}.

We would like to understand the limiting behavior of $$\wtY_i:= \rmH_i\ep_i\Gamma/\Gamma$$ in $\rmG/\Gamma$. So far we know that 
\begin{itemize}
    \item[1.]  $\lim\wtY_i$ is a closed $ \rmH_{0}$-invariant set;
    \item[2.] $\lim\wtY_i \supset \wtY :=  \rmH_{0}\Gamma/\Gamma$;
    \item[3.] $\lim\wtY_i \subset \rmK_0  \rmH_{0} \Gamma /\Gamma$.
\end{itemize}

Assume $[k_1g_1]_{\Gamma} \in \lim\wtY_i $ for some $k_1\in\rmK_0$ and $g_1 \in  \rmH_{0}$. Then by item 1. above, 
$\lim\wtY_i \supset \rmH_{0} k_1g_1 \Gamma /\Gamma$ and in particular 
\begin{equation*}
\begin{aligned}
        &\rmK_0 \bs \rmK_0  \rmH_{0}\Gamma/\Gamma=
    \lim \rmK_0 \bs \rmK_0 \rmH_i \ep_i \Gamma/\Gamma \supset 
    \rmK_0 \bs \rmK_0  \rmH_{0} k_1 g_1 \Gamma/\Gamma.\\
    \implies &
    \rmK_0 \bs \rmK_0  \rmH_{0}g_1\Gamma/\Gamma\supset  \rmK_0 \bs \rmK_0 (k_1^{-1}\rmH_0 k_1) g_1 \Gamma/\Gamma.
\end{aligned}
\end{equation*}
Thus $\fraks_0$ contains $\Ad(k_1^{-1})\cdot \fraks_0$. 
And since they have the same dimensions,  $\fraks_0=\Ad(k_1^{-1})\cdot \fraks_0$, which implies that $k_1\in N_{\rmG}(\rmH_0)$, the normalizer of $\rmH_0$ in $\rmG$. Therefore item 3. above is upgraded to 
\begin{itemize}
    \item[4.] $\lim\wtY_i \subset (\rmK_0\cap N_{\rmG}(\rmH_0))   \rmH_{0} \Gamma /\Gamma \subset N_{\rmG}(\rmH_0)\Gamma/\Gamma$.
\end{itemize}

Let $\rmH_i':= \ep_i^{-1}\rmH_i\ep_i$ and $\frakh'_i:=\Ad(\ep_i^{-1})\frakh_i$.
Decompose $\frakh'_{i}=\fraka'_{i}\oplus \frakm'_{i}$ into an abelian ideal $\fraka_{Y_i}$ and a semisimple ideal $\frakm_{Y_i}$. 
Both $\fraka'_{i}$ and $\frakm'_{i}$ are $\iota_0$-stable. 
Write $\rmA'_i$ and $\rmM'_i$ for the associated Lie subgroups.
By Borel density lemma (use the version in \cite[Corollary 4.2]{DaniBorelDesity}), $\bmH_i'$ is defined over $\Q$. Thus $\bmM'_i$ and $\bmA'_i$ are also defined over $\Q$. As $Y_i$ has finite volume, it follows that $\rmM'_i\Gamma/\Gamma$ and $\rmA'_i\Gamma/\Gamma$ have finite volume.

By \cite[Theorem 1.1]{MozShah95}, there exists a connected $\R$-split subgroup $\bmF$  of $\bmG$ such that $[\rmF]_{\Gamma}$ has finite volume, $[\rmM_i']_{\Gamma}$ converges to $[\rmF]_{\Gamma}$ and moreover,
there exists $\delta_i\in \rmG$ converging to $id$ such that $[\delta_i\rmM_i ]_{\Gamma}$ is contained in $[\rmF]_{\Gamma}$ for $i$ large enough.
From the latter, it can be shown that $\rmM'_i$ is contained in $\rmF$ for $i$ large enough.
On the other hand, the limit of $[\rmM'_i ]_{\Gamma}$ is contained in $[N_{\rmG}(\rmH_0)]_{\Gamma}$, thus $\rmF$ is contained in $N_{\rmG}(\rmH_0)$. As $\rmF$ is connected and the Lie algebra of $N_{\rmG}(\rmH_0)$ is the same as that of $Z_{\rmG}(\rmH_0)^{\circ}\rmH_0$. We conclude that $\rmF$ is contained in $Z_{\rmG}(\rmH_0)^{\circ}\rmH_0$. But $\rmF$ is semisimple and contains $\rmM'_i$, thus we must have $\rmM_0$, the semisimple part of $\rmH_0$, is exactly equal to $\rmF$. Hence $\rmM_i'=\rmF=\rmM_0$.

Thus we have seen that (write $\wtY_i':= \ep_i^{-1}\wtY_i$)
\begin{itemize}
    \item[5.] $\wtY'_i=  \rmA'_i \rmM_0\Gamma/\Gamma \subset Z_{\rmG}(\rmM_0)^{\circ} \rmM_0\Gamma/\Gamma$ for $i$ large enough.
\end{itemize}
So to find the limit of $\wtY'_i$, it suffices to consider
\begin{equation*}
    \lim [\rmA'_i\rmM_0 ]_{\Gamma_{Z}} \text{ inside }
    Z_{\rmG}(\rmM_0)^{\circ}\rmM_0/ \Gamma \cap (  Z_{\rmG}(\rmM_0)^{\circ}\rmM_0).
\end{equation*}

Now we have all the ingredients to conclude the proof. Some definitions and notations are introduced to ease the argument.

Let $\rmL:= Z_{\rmG}(\rmM_0)^{\circ}\rmM_0$ and $\Gamma_{\rmL}:= \Gamma \cap \rmL$. 
Let $\pi: \rmL \to \rmL/\rmM_0$ be the natural quotient map and it induces $\pi' : \rmL/\Gamma_{\rmL} \to \pi(\rmL)/\pi(\Gamma_{\rmL})$. Note that $\pi(\Gamma_{\rmL})$ is  still a lattice in $\pi(\rmL)$.
Let $\rmK_{\rmZ}:= \rmK_0 \cap Z_{\rmG}(\rm\rmH_0)^{\circ} = ( \rmK_0 \cap Z_{\rmG}(\rm\rmH_0))^{\circ}$. 
For an element (or a subset) $x$ of $\pi(\rmL)$, as before, we let ${}_{\pi(\rmK_{\rmZ})}[x]_{\pi(\Gamma_{\rmL})}$ be its image in $\pi(\rmK_{\rmZ})\bs\pi(\rmL)/\pi(\Gamma_{\rmL})$. Other similar notations are also defined . Note that since $\rmA_0$ commutes with $\rmK_{\rmZ}$, $\pi(\rmA_0)$ acts from the left on $\pi(\rmK_{\rmZ})\bs\pi(\rmL)/\pi(\Gamma_{\rmL})$.

Let $\{x_1,...,x_c\}\subset N_{\rmG}(\rmH_0)\cap \rmK_0$ be a set of representatives for the quotient $N_{\rmG}(\rmH_0)\cap \rmK_0/ (N_{\rmG}(\rmH_0)\cap \rmK_0)^{\circ}$.
In light of item 4. and 5. above, only those $x_s$'s contained in $\rmL\Gamma$ are interesting to us. By rearranging the order, we find $1\leq c_0\leq c$ such that for each $1\leq s \leq c_0$, there exists $l_{x_s}\in \rmL$, $\gamma_{x_s}\in\Gamma$ such that \begin{equation*}
    x_s= l_{x_s} \gamma_{x_s}.
\end{equation*}

Now we start the argument.
Item 4. above implies that 
\begin{equation*}
\begin{aligned}
          \lim\wtY_i'=\lim \rmA_i' \rmM_0\Gamma/\Gamma \subset& 
           \bigcup_{s=1,..,c_0}    \rmH_{0} (N_{\rmG}(\rmH_0)\cap \rmK_0)^{\circ}l_{x_s}\Gamma /\Gamma\\
           =&  \bigcup_{s=1,..,c_0}    \rmH_{0} \rmK_{\rmZ} l_{x_s}\Gamma /\Gamma
\end{aligned}
\end{equation*}
because $ \rmH_{0}\rmK_{\rmZ} =  \rmH_{0}(Z_{\rmG}(\rmH_0)\cap \rmK_0)^{\circ}=  \rmH_{0} (N_{\rmG}(\rmH_0)\cap \rmK_0)^{\circ}$.

As $(\fraka'_i)$ converges to $\fraka_0$ and they correspond to maximal $\R$-split subtori of $\bmZ_{\bmG}(\bmM_Y)$, there exists $\ep'_i \to id$ in $\rmL$ such that
\begin{equation*}
    \ep'_i\rmA'_i\ep_i'^{-1}=\rmA_0
\end{equation*}
where we are using the fact that an orbit of an algebraic group is open in its closure, see \cite[2.3.3]{Spr98}.  So we have 
\begin{equation*}
    \lim \rmA_0 \ep_i' \rmM_0\Gamma/\Gamma \subset 
     \bigcup_{s=1,..,c_0}    \rmH_{0} \rmK_{\rmZ} l_{x_s}\Gamma /\Gamma.
\end{equation*}
Now this is a convergence inside $\rmL\Gamma/\Gamma$, which we identify with $\rmL/\Gamma_{\rmL}$. Thus 
\begin{equation*}
   \lim  [ \rmA_0 \ep_i' \rmM_0]_{\Gamma_{\rmL}} \subset 
   \bigcup_{s=1,..,c_0} [  \rmH_{0} \rmK_{\rmZ} l_{x_s}]_{\Gamma_{\rmL}} \;\;\text{inside}\;\;\rmL/\Gamma_{\rmL} .
\end{equation*}
By applying $\pi'$ we have
\begin{equation*}
   \lim  [ \pi(\rmA_0 \ep_i') ]_{\pi(\Gamma_{\rmL})} \subset 
   \bigcup_{s=1,..,c_0} [\pi( \rmA_0 \rmK_{\rmZ} l_{x_s})]_{\pi(\Gamma_{\rmL})}\;\;\text{inside}\;\;
   \pi(\rmL)/\pi(\Gamma_{\rmL} ).
\end{equation*}
By further projecting to $\pi(\rmK_{\rmZ})\bs \pi(\rmL)/\pi(\Gamma_{\rmL} )$, we have
\begin{equation*}
\begin{aligned}
    \lim \pi(\rmA_0) {}_{\pi(\rmK_{\rmZ})}[ \pi(\ep_i') ]_{\pi(\Gamma_{\rmL})} 
   \subset 
   \bigcup_{s=1,..,c_0} \pi(\rmA_0)
  {}_{\pi(\rmK_{\rmZ})} [\pi( l_{x_s})]_{\pi(\Gamma_{\rmL})}.
\end{aligned}
\end{equation*}
As the right hand side is a union of closed orbits of $\pi(\rmA_0)$, we may write it as a disjoint union by identifying certain indices,
\begin{equation*}
     \lim \pi(\rmA_0) {}_{\pi(\rmK_{\rmZ})}[ \pi(\ep_i') ]_{\pi(\Gamma_{\rmL})} 
   \subset 
   \bigsqcup \pi(\rmA_0)
  {}_{\pi(\rmK_{\rmZ})} [\pi( l_{x_s})]_{\pi(\Gamma_{\rmL})}.
\end{equation*}

As each of $\pi(\rmA_0)
  {}_{\pi(\rmK_{\rmZ})} [\pi( l_{x_s})]_{\pi(\Gamma_{\rmL})}$ is a compact set, there exists bounded neighborhood $\calN_s$  such that the closures of $\calN_s$ as $s$ varies are disjoint. But for all $i$, $\pi(\rmA_0) {}_{\pi(\rmK_{\rmZ})}[ \pi(\ep_i') ]_{\pi(\Gamma_{\rmL})} $ are connected, so they are contained in exactly one $\calN_i$ for $i$ large enough and therefore
  \begin{equation*}
      \lim \pi(\rmA_0) {}_{\pi(\rmK_{\rmZ})}[ \pi(\ep_i') ]_{\pi(\Gamma_{\rmL})} 
   \subset 
    \pi(\rmA_0)
  {}_{\pi(\rmK_{\rmZ})} [id]_{\pi(\Gamma_{\rmL})}.
  \end{equation*}
  
  Now take a small neighborhood $\calN$ of 
  \begin{equation*}
      \pi(\rmA_0)
  {}_{\pi(\rmK_{\rmZ})}[id]_{\pi(\Gamma_{\rmL}\cap \rmA_0) } \subset \calN \subset
  \pi(\rmK_{\rmZ}) \bs \pi(\rmL)/
  \pi(\Gamma_{\rmL}\cap \rmA_0) .
  \end{equation*}
  Let $p: \pi(\rmK_{\rmZ}) \bs \pi(\rmL)/
  \pi(\Gamma_{\rmL})  \to
  \pi(\rmK_{\rmZ}) \bs \pi(\rmL)/
  \pi(\Gamma_{\rmL}\cap \rmA_0) $ denote the natural projection.
  Choose $\calN$ small enough such that $p$ restricted to $\calN$ is a homeomorphism onto its image.
  Then for $i$ large enough, 
  \begin{equation*}
      \pi(\rmA_0) {}_{\pi(\rmK_{\rmZ})}[ \pi(\ep_i') ]_{\pi(\Gamma_{\rmL})}
      \subset p(\calN).
  \end{equation*}
  Hence for $i$ large enough,
  \begin{equation*}
      \pi(\rmA_0) {}_{\pi(\rmK_{\rmZ})}[ \pi(\ep_i') ]_{\pi(\Gamma_{\rmL}\cap \rmA_{Y})}
      \subset \calN.
  \end{equation*}
  This shows, in particular, that $\pi(\rmA_0\ep_i')$ is bounded in $\pi(\rmL)/\pi(\rmA_0)$. Consequently, 
  \begin{equation*}
      \pi(\ep_i'^{-1}\rmA_0\ep_i') = \pi(\rmA_0),\quad
      \rmH'_i= \rm\rmH_0,\quad Y_i= \rmK_0 \bs \rmK_0\ep_i \rm\rmH_0 \Gamma/\Gamma.
  \end{equation*}
   But by item 2 of Lemma \ref{lemmaRigidisAlgebraic}, $\rm\rmH_0$ only has one totally geodesic orbit on $\rmK_0 \bs \rmG$. So $  \rmK_0 \bs \rmK_0\ep_i \rm\rmH_0=  \rmK_0 \bs \rmK_0\rm\rmH_0$ and the proof completes.
   
\subsection{Proof of Theorem \ref{thmGeometric}}   

The proof follows the same line as in \cite{OhETDS2004}.

By Lemma \ref{lemmaRigidisAlgebraic} and use the notation there, $Y= \rmK_0 \bs \rmK_0\rmH_ig_{Y}\Gamma/\Gamma$ for some $i\in\{1,...,s\}$ and $g_{Y}\in \rmG$. By Theorem \ref{thmMain}, we may assume $g_{Y}$ belongs to a fixed compact set $\calC \in\rmG$. Depending on $\calC$, we find a number $C_{12}>0$ such that 
\begin{equation*}
    \Vol(h^{-1}\rmH_ig\Gamma/\Gamma)<C_{12} \cdot \Vol(\rmH_ig\Gamma/\Gamma)  
\end{equation*}
for all $i\in\{1,...,s\}$ and $g\in\rmG$ such that $\rmH_ig\Gamma/\Gamma$ has finite volume. Let $C_{11}$ be as in Lemma \ref{lemmaVolBounded}.
Thus 
\begin{equation*}
    \Vol(g_Y^{-1}\rmH_ig_Y\Gamma/\Gamma) \leq C_{12} \Vol(\rmH_ig_Y\Gamma/\Gamma)   \leq 
    C_{11}C_{12} \Vol(Y).
\end{equation*}
Thus for a fixed $T>0$,  $\Vol(g_Y^{-1}\rmH_ig_Y\Gamma/\Gamma)$ as $Y$ varies in $ \calT\calG^{N,\rigid}_{\leq T}$ is bounded. By \cite[Theorem 5.1]{DanMar93}, the set 
\begin{equation*}
    \left\{
    g_Y^{-1}\rmH_ig_Y\cap \Gamma\; \middle\vert\;
    Y\in \calT\calG^{N,\rigid}_{\leq T}
    \right\}
\end{equation*}
is finite. To recover $Y$ from $ g_Y^{-1}\rmH_ig_Y\cap \Gamma$, it suffices to note that $g_Y^{-1}\rmH_ig_Y\cap \Gamma$ is Zariski dense in $g_Y^{-1}\rmH_i g_Y$ and that 
$\rmK_0 \bs \rmK_0 \rmH_i g_{Y}$ is the unique orbit of $g_Y^{-1}\rmH_i g_Y$ on $\rmK_0\bs \rmG$ that is totally geodesic. So we are done.

\bibliographystyle{amsalpha}
\bibliography{ref}

\end{document}